\documentclass[12pt,oneside]{amsart}
\usepackage{amsxtra}
\usepackage{amsopn}
\usepackage{amsmath,amsthm,amssymb}
\usepackage[hypertex]{hyperref} 

\title[New  HKT manifolds arising from quaternionic representations]{New  HKT manifolds arising from \\quaternionic representations}
\author{M. L. Barberis}
\author{ A. Fino}
\thanks{Research partially supported by
Conicet, ANPCyT, Secyt-UNC (Argentina) and by MIUR, GNSAGA (Italy)}

\address{Laura Barberis: FaMAF-CIEM, Universidad
Nacional de C\'ordoba\\5000 C\'ordoba, Argentina}
\email{barberis@mate.uncor.edu}
  \address{Anna Fino: Dipartimento di Matematica, Universit\`a di Torino\\
Via Carlo Alberto 10, 10123 Torino, Italy}
\email{annamaria.fino@unito.it}

\newtheorem{teo}{Theorem}[section]
\newtheorem{remark}{Remark}[section]
\newtheorem{prop}{Proposition}[section]
\newtheorem{corol}{Corollary}[section]
\newtheorem{lemm}{Lemma}[section]
\newtheorem{example}{Example}[section]

\newcommand{\beq}{\begin{equation}}
\newcommand{\eeq}{\end{equation}}
\newcommand{\bqn}{\begin{eqnarray}}
\newcommand{\eqn}{\end{eqnarray}}
\newcommand{\bqne}{\begin{eqnarray*}}
\newcommand{\eqne}{\end{eqnarray*}}

\newcommand{\R}{{\Bbb R}}
\renewcommand{\H}{{\Bbb H}}
\newcommand{\C}{{\Bbb C}}

\begin{document}
\begin{abstract} We give a procedure for constructing an  $8n$-dimensional HKT Lie algebra starting from a
$4n$-dimensional  one by using a quaternionic representation of the
latter. The strong (respectively, weak, hyper-K\"ahler, balanced)
condition is preserved by our construction. As an application of our
results we obtain a new compact HKT manifold with holonomy in
$SL(n,\Bbb H)$ which is not a nilmanifold. We  find in addition new
compact  strong HKT manifolds.

We also show that every K\"ahler Lie algebra equipped with a flat,
torsion-free complex connection gives rise to an HKT Lie algebra. We
apply this method to two distinguished  $4$-dimensional K\"ahler Lie
algebras, thereby obtaining two conformally balanced HKT metrics in
dimension $8$.

Both techniques prove to be an effective tool for giving the explicit expression of the corresponding HKT metrics.
\end{abstract}

\maketitle

\section{Introduction}

A hyper-Hermitian structure on a $4n$-dimensional manifold $M$ is
given by a hypercomplex structure $\{ J_ \alpha \}$, $\alpha =
1,2,3$ and a Riemannian metric $g$ compatible with $J_{\alpha}$, for
any $\alpha$.  The hyper-Hermitian   manifold  $(M, \{ J_ \alpha \},
g)$ is said to be  {\em hyperk\"ahler with torsion (HKT)} \cite{HP}
if there exists a hyper-Hermitian connection  $\nabla^B$ whose
torsion tensor is a $3$-form, i.e. if
$$
\nabla^B g =0, \quad  \nabla^B J_{\alpha} =0, \, \alpha = 1,2,3,
\quad c (X,Y, Z) = g (X, T^B(Y, Z)) \; {\mbox{ is a $3-$form}},
$$
where  $T^B$ is  the torsion of $\nabla^B$. An HKT structure is
called {\em strong} or {\em weak} according  to the fact that the
$3$-form $c$  is closed or not.  HKT geometry is a natural
generalization of hyper-K\"ahler geometry, since when $c =0$ the
connection $\nabla^B$ coincides with the Levi-Civita connection.

 On any Hermitian
manifold there exists a unique Hermitian connection whose torsion
tensor is a $3$-form. Such a connection is called in Hermitian
geometry the Bismut connection \cite{Bi} or KT connection.
In the case of an HKT manifold the three Bismut connections
associated to the Hermitian structures $(J_{\alpha}, g)$ coincide
and   this connection is said to be an HKT connection.  In terms of
the associated K\"ahler forms
$$
\omega_{\alpha} (X, Y) = g (J_{\alpha} X, Y), \quad \alpha =1,2,3,
$$
the HKT condition becomes equivalent to
\begin{equation}\label{hkt}
J_1 d   \omega_1 = J_2 d   \omega_2 = J_3 d   \omega_3 ,
\end{equation}
or to $\overline \partial_{J_1} (\omega_2 - i \omega_3) =0$
\cite{GP}. It has been shown in \cite{CS} that if $(M, \{ J_
\alpha\}, g)$ is almost hyper-Hermitian, then condition \eqref{hkt}
implies the integrability of $J_ \alpha, \; \alpha =1,2,3.$

Hyper-Hermitian connections with totally skew-symmetric torsion are
present in many branches of theoretical and mathematical  physics.
For instance,  the previous connections appear on  supersymmetric
sigma models  with  Wess-Zumino term \cite{GHR, HP, HP2}  and in
supergravity theory \cite{GPS, PT}.

The first examples of strong HKT structures were found on reductive
Lie  groups   with compact semisimple factor by using the
hypercomplex structure constructed  by Joyce in \cite{Jo} (see for
instance \cite{GP,SSTV}). These Lie groups are   also  examples of generalized hyper-K\"ahler manifolds. This type of structures
was  introduced by Gates, Hull and R\"ocek in
\cite{GHR} and studied in more detail
in \cite{Br}.

 In \cite{BS} it was proved that,  as in
the hyper-K\"ahler case,  locally any HKT metric admits an HKT
potential and in \cite{GPP}     HKT reduction has been studied in
order to construct new examples. Non-homogeneous examples of compact
HKT  manifolds  have been constructed  in \cite{Ve2} by considering
the total space of a hyperholomorphic  bundle over an HKT  manifold.

Some geometrical and topological properties have been studied. For
instance, a   simple characterisation of HKT  geometry in terms of
the intrinsic torsion of the $Sp(n) Sp(1)$-structure   was obtained
in \cite{CS}. A version of  Hodge theory for HKT manifolds has been
given   in  \cite{Ve3} by exploiting a remarkable analogy between
the de Rham complex of a K\"ahler manifold and the Dolbeault complex
of an HKT manifold. More recently, in \cite{Ve4} balanced HKT
metrics are studied, showing that the  HKT metrics are precisely the
quaternionic Calabi-Yau metrics defined in terms of the quaternionic
Monge-Amp\`ere equation.  Moreover, by \cite{Ve4}  a balanced HKT manifold has Obata connection with holonomy in $SL (n,  \H).$

In $4$ dimensions any  hyper-Hermitian manifold is HKT, but in
higher dimensions  this is no longer true. In fact, there exist
hypercomplex manifolds of dimension $\geq 8$ which do not admit any
HKT metric compatible with the hypercomplex structure \cite{FG,BDV}.
These manifolds are nilmanifolds, i.e. they are compact quotients of
nilpotent Lie groups by co-compact discrete subgroups. Strong KT
geometry on $6$-dimensional nilmanifolds has been studied in
\cite{FPS,Ug}.

It is known  that  on a Lie algebra $\frak g$, a hyper-Hermitian
structure with an abelian hypercomplex  structure $\{ J_{\alpha}
\}$,  i.e. such that
$$
[J_{\alpha} X, J_{\alpha} Y] = [X, Y],
$$
for any  $X, Y\in \frak g$ and $\alpha =1,2,3$,
  is always  weak  HKT, so many  HKT manifolds can be constructed by considering  abelian hypercomplex structures on solvable Lie algebras~\cite{DF}.

HKT geometry on nilmanifolds was investigated in \cite{DF}.
Subsequently in \cite{BDV}, it has been shown that a hyper-Hermitian
structure on a nilmanifold is HKT if and only if the hypercomplex
structure is abelian. In this work we provide  a counterexample to
this result for HKT solvmanifolds  (\S\ref{new_ex}). HKT structures
on nilmanifolds are always  weak,  and by \cite{Ve} the holonomy of
the Obata connection is contained   in $SL(n,
 \H)$,
 since for complex nilmanifolds the canonical bundle is trivial as a  holomorphic line bundle \cite{CG,BDV}.
As pointed out in \cite{Ve4}, the only known examples of  compact
HKT manifolds with holonomy in $SL(n, \Bbb H)$ are those
nilmanifolds admitting abelian hypercomplex structures.  In the
present paper we obtain  a new compact HKT manifold with holonomy in
$SL(n, \Bbb H)$ which is not a nilmanifold (\S\ref{new_ex}).
The
existence of strong HKT structures on solvmanifolds is an open
problem.

It has been shown in \cite{BD} that given a Hermitian or
hyper-Hermitian Lie algebra, its tangent Lie algebra with respect to
a suitable connection admits a natural hyper-Hermitian structure.
 In the paper we consider both cases separately  and show that the
tangent Lie algebra of an HKT Lie algebra may admit an HKT
structure. In this way  we can
construct  a family of new compact strong HKT manifolds, compact
quotients   of Lie groups which are neither reductive nor solvable
(Theorem \ref{compactHKT}).


The construction is made starting from a $4n$-dimensional HKT  Lie
algebra $\mathfrak g$ with a flat hyper-Hermitian connection
(Theorem \ref {construction}). This connection turns out to be  a
quaternionic representation of $\mathfrak g$ on $\H^n$. The weak,
strong, hyper-K\"ahler or balanced condition is preserved by the
construction. This result is generalised in
Theorem~\ref{generalization} starting from any quaternionic
representation of $\mathfrak g$ on $\H^q$, for arbitrary $q$, and it
can be viewed as an application of the general results obtained in
\cite{Ve2} for hyperholomorphic bundles over HKT manifolds.

We apply the previous procedure to the known strong HKT reductive
Lie algebras.   One of the compact strong  HKT manifolds  is the
product  of  $S^1$ by the  $7$-dimensional manifold with a weakly
integrable generalized $G_2$-structure found in \cite{FT}.  Starting
with an abelian hypercomplex structure on $\mathfrak g$  our
construction does not give an abelian  hypercomplex structure on the
new Lie algebra, unless the quaternionic representation is trivial
(Theorem~\ref{condab}).

On the other hand, for the hyper-Hermitian Lie algebras  constructed
in \cite{BD} starting from a Hermitian Lie algebra  with a flat and
torsion-free complex connection, we  prove that  if  the
hyper-Hermitian structure is HKT, then  the starting Hermitian Lie
algebra is necessarily K\"ahler  (Theorem \ref{construction2}).  The
hyper-Hermitian Lie algebras obtained in this way have flat Obata
connection. As an application of this result we find two
$8$-dimensional   Lie algebras with a  conformally balanced weak HKT structure, so in
particular these metrics are not hyper-K\"ahler.

\section{ Preliminaries}

Let $M$ be a $2n$-dimensional manifold.
We recall that an almost  complex structure $J$ on M  is
 integrable   if and only if
  the Nijenhuis tensor
$$N(X,Y)= J([X,Y]-[JX,JY])-([JX,Y]+[X,JY])$$
 vanishes for all vector fields  $X,Y$. In this case   $J$  is called a complex structure on $M$.

A  Riemannian metric $g$ on a complex manifold $(M, J)$ is said to
be  Hermitian if it is compatible with $J$, i.e. if  $g (JX, JY) = g
(X, Y)$, for any $X, Y$. We recall from \cite{Bi} that on any
Hermitian manifold $(M, J, g)$ there exists a unique  connection
$\nabla^B$, called the Bismut connection, such that
$$
\nabla^B J = \nabla^B g = 0
$$
and its torsion  tensor
$$c (X, Y, Z) = g (X, T^B (Y, Z))$$ is totally skew-symmetric, where
$T^B$ is the torsion of  $\nabla^B$.  The geometry associated to the
Bismut connection is called KT geometry and when $c =0$ it coincides
with the usual K\"ahler geometry.  A Hermitian metric $g$ on   a
complex manifold $M$  is called {\em{balanced}}
 if   the Lee form $\theta =  J d^* \omega$    vanishes or equivalently if $d^* \omega=0$, where $d^*$ is the adjoint of $d$ with respect to $g$ and $\omega$ is the associated K\"ahler form. Moreover, by
\cite{IP}  if  $\{e_1, \ldots,   e_{2n}\}$  is an orthonormal basis
of the tangent space $T_p M$, then
\begin{equation} \label{Leeform}
\theta _p(v) = - \frac 12  \sum_{i = 1}^{2n}  c (Jv, e_i, J e_i),
\end{equation}
for any tangent vector  $v \in T_p M$.

A hyper-Hermitian structure $(\{ J_{\alpha} \}, g)$ on a $4n$-dimensional manifold $M$  is given by  a  hypercomplex structure, i.e. a triple of complex structures
  $\{J_{\alpha}\}_ {\alpha=1,2,3}$ satisfying the quaternion relations
$$J_{\alpha}^2= -\text{id},\quad \alpha=1,2,3,\qquad J_1J_2=-J_2J_1=J_3,$$
and by a Riemannian metric $g$ compatible with $J_{\alpha}$, for any
$\alpha$. Given a hypercomplex  structure $\{ J_{\alpha} \}$  on $M$
there exists a unique torsion-free connection $\nabla^O$, called the
Obata connection,  such that $\nabla^O J_{\alpha} = 0$, for $\alpha
= 1, 2, 3$ (cf. \cite{Ob}).
%
%

 If on a  hyper-Hermitian manifold $(M, \{
J_{\alpha} \}, g)$  there exists a
hyper-Hermitian connection such that its torsion tensor $c$ is a
$3$-form, then the manifold $M$  is
  called   {\em hyper-K\"ahler with torsion} (HKT). This is equivalent to the fact  that the three Bismut connections associated to
  each Hermitian structure $(J_{\alpha}, g)$
  coincide.  If $c =0$ the Bismut connection is equal to both the Levi-Civita connection and the  Obata connection, since the
  manifold is hyper-K\"ahler.  HKT structures are called {\em {strong}} or {\em {weak}} depending on whether the
  torsion $c$ is closed or not. In the case of  an HKT manifold the Lee  forms  associated to the
  Hermitian structures  $(J_\alpha, g)$, $\alpha = 1,2,3$, coincide (see \cite{IP}).
Moreover, in general on an HKT manifold  the Ricci tensor of the Bismut connection is not  symmetric; by  \cite{IP} it is symmetric
if and only if  the torsion  $3$-form $c$ is co-closed, i.e. if and only if $d^* c =0$.

    A Riemannian  manifold  $(M,  g)$  is called
  {\em  generalized hyper-K\"ahler} or a $(4,4)$-manifold  if it has a pair
 of  strong HKT structures $(\{J^+_{\alpha} \}, g)$  and $(\{J^-_{\alpha} \}, g)$ for which $c^-= - c^+$,
 where $c^{\pm}$ denotes the torsion $3$-form  of the hyper-Hermitian connection  associated to the   HKT  structure $(\{J^{\pm}_{\alpha} \}, g)$.
These manifolds have been introduced  in \cite{GHR} and further
studied in  \cite{Gu, Br}.

  In this paper we  consider  Lie algebras  endowed with HKT  structures  which induce  left-invariant HKT structures  on  the corresponding simply connected Lie groups.

  Let $\mathfrak g$ be a Lie algebra  with an  (integrable)  complex
structure $J$  and an inner product $g$ compatible with $J$. If the
associated K\"ahler form $\omega (X,Y)= g(JX,Y)$ satisfies
$d\omega=0$, where
$$
d \omega(X, Y, Z) = - \omega ([X, Y], Z) - \omega([Y, Z], X) -
\omega([Z, X], Y),
$$
for any $X, Y, Z \in {\mathfrak g}$,
 the Hermitian Lie algebra  $(\mathfrak g , J,g)$ is  K\"ahler. Equivalently, $(\mathfrak g , J,g)$ is  K\"ahler
if and only if  $\nabla^{g} J =0$, where $\nabla^{g} $ is the Levi-Civita connection of $g$, which
  can be computed by
\begin{eqnarray}\label{LC}
2g(\nabla^{g} _XY,Z) &=&g([X,Y],Z) -g([Y,Z],X) + g([Z,X],Y),
\end{eqnarray}
for any $X,Y,Z$ in $\mathfrak{g}$.

Given a  hyper-Hermitian structure  $( \{ J_{\alpha} \}, g)$ on $\mathfrak g$,
   when the associated K\"ahler forms
$\omega_{\alpha}$ are closed, the hyper-Hermitian Lie algebra
$(\mathfrak g , \{ J_{\alpha} \}, g)$ is hyper-K\"ahler. This is
equivalent to the condition $\nabla^{g}   J_{\alpha}=0$, $\alpha
=1,2,3$. We point out  that a Lie group with a left-invariant
hyper-K\"ahler structure is necessarily flat since a hyper-K\"ahler
metric is Ricci flat and in the homogeneous case, Ricci flatness
implies flatness (see \cite{AK}). A characterisation of
hyper-K\"ahler Lie groups has been carried out  in \cite{BDF} in
order to obtain new complete hyper-K\"ahler metrics.

 For a  Lie  group $G$ with a
left-invariant hyper-Hermitian structure $(\{ J_{\alpha} \}, g)$, it was shown in \cite{DF}  that $(\{ J_{\alpha} \}, g)$
  is HKT  if and only if
  \begin{equation}
\label{HKTcondition}
\begin{array} {l}
g ([J_{\alpha} X, J_{\alpha} Y], Z) + g ([J_{\alpha} Y, J_{\alpha}
Z], X)
  + g ([J_{\alpha} Z, J_{\alpha} X], Y)\\ [4pt]
  = g ([J_{\beta} X, J_{\beta} Y], Z)
   + g ([J_{\beta} Y, J_{\beta} Z], X)
   + g ([J_{\beta} Z, J_{\beta} X], Y),
   \end{array}
\end{equation}
for any $X, Y, Z $ in the Lie algebra ${\mathfrak g}$ of $G$ and for any $\alpha$, $\beta \in \{ 1,
2, 3\}$. The Bismut connection $\nabla^B$ on $G$  is given by
the following equation
\begin{equation} \label{Bismut}
\begin{array} {l}
g (\nabla^B_X Y, Z) = \frac{1}{2} \{ g ([X, Y] - [J_{\alpha} X,
J_{\alpha} Y], Z)\\ [4pt]
  - g ([Y, Z] + [J_{\alpha} Y, J_{\alpha} Z], X)
  + g ([Z,X]
- [J_{\alpha} Z, J_{\alpha} X], Y)\},
\end{array}
\end{equation}
for any $X, Y, Z \in {\mathfrak g}$ and $\alpha \in \{ 1, 2, 3\}$.

\

\section{Construction of HKT   structures on tangent  Lie algebras}\label{Section}

   Let $\{ J_{\alpha} \}$ be a hypercomplex structure on a Lie
algebra ${\mathfrak g}$ and assume that $D$ is a flat connection on
$\mathfrak g$ such that $D J_{\alpha} =0$, for any $\alpha= 1,2,3$.
In other words,
\[ D:\mathfrak g \rightarrow \mathfrak{gl}(n,\Bbb H )\]
is a Lie algebra homomorphism, where $\dim \mathfrak g =4n$.
  Consider the tangent Lie algebra $T_{D} \,  \frak g :=
{\mathfrak g} \ltimes_{D} {\mathfrak g}$ with the following Lie
bracket: \begin{equation}\label{t_D}  [(X_1,X_2), (Y_1,Y_2)]=([X_1,
Y_1] , D _{X_1} Y_2 - D_{Y_1} X_2 )
\end{equation}
and  hypercomplex structure:
\begin{equation}  \label{hypercomplex}
{\tilde J}_1(X_1, X_2)= (J_1 X_1, J_1 X_2), \quad  {\tilde J}_2(X_1, X_2)= (J_2 X_2,  J_2 X_1),
\quad  {\tilde J}_3 = {\tilde J}_1 {\tilde J}_2.
\end{equation}

Since $D$ is flat, the Lie bracket \eqref{t_D} on $T_{D} \,  \frak
g$ satisfies the Jacobi identity. The integrability of the complex
structure ${\tilde J}_{\alpha}$ on $T_{D} \,  \frak g$ follows from
the fact that $J_{\alpha}$ is integrable and parallel with respect
to $D$ (see \cite[Proposition~3.3]{BD}). We show next that the Obata
connection $ {\tilde \nabla}^O$ of $\{ {\tilde J}_{\alpha} \}$ is
given in terms of  the Obata connection $\nabla ^O$ of $\{
J_{\alpha} \}$ and the flat connection $D$.

\begin{lemm}\label{lem_Ob} The Obata connection ${\tilde \nabla}^O$ of the hypercomplex  Lie
algebra $(T_{D} \, \frak g , \{ {\tilde J}_{\alpha } \})$ is related
to the Obata connection $\nabla^O$ of  $({\mathfrak g}, \{
J_{\alpha} \} )$ and the flat connection $D$ by
\begin{equation}\label{twoObata}
{\tilde \nabla}^O_{(X_1, X_2)} (Y_1, Y_2) = (\nabla^O_{X_1} Y_1,
D_{X_1} Y_2),
\end{equation}
for any $(X_1, X_2), (Y_1, Y_2) \in T_{D} \, \frak g$. Therefore,
${\tilde \nabla}^O$ and ${ \nabla}^O$ have the same infinitesimal
holonomy. In particular, $\tilde \nabla^O$ is flat if any only if
$\nabla^O$ is flat.
\end{lemm}
\begin{proof}
The connection ${\tilde \nabla}^O $ defined by \eqref{twoObata} is
torsion-free and satisfies ${\tilde \nabla}^O {\tilde J}_{\alpha}
=0$ for $\alpha =1,2,3$, therefore ${\tilde \nabla}^O  $  is the
Obata onnection corresponding to $\{ {\tilde J}_{\alpha } \}$.
Moreover, if $R^O$ and ${\tilde R}^O$ denote the curvature tensors
of $\nabla^O$ and $\tilde \nabla^O$, respectively, we have the
following relation:
$$
\begin{array} {lcl}
{\tilde R}^O_{(X_1, X_2), (Y_1, Y_2)} (Z_1, Z_2) &= &[{\tilde  \nabla}^O_{(X_1, X_2)},{\tilde  \nabla}^O_{(Y_1, Y_2)}] (Z_1, Z_2) -
{\tilde  \nabla}^O_{[(X_1, X_2), (Y_1, Y_2)]}  (Z_1, Z_2)\\[3pt]
& = &(R^O_{X_1, Y_1} Z_1, 0).
\end{array}
$$
\end{proof}

We  study next a necessary and sufficient condition for $\{ {\tilde
J}_{\alpha } \}$ to be an abelian hypercomplex structure, since it
is known that such structures on a Lie algebra give rise to weak HKT
structures, see for instance \cite{GP,DF,BDV,Fi}. In fact, we show
that a  hypercomplex structure $\{J_{\alpha} \}$ on $\mathfrak g$
induces an abelian hypercomplex structure  on $T_{D} \, \mathfrak g$
if and only if $\{J_{\alpha} \}$ is abelian and $D =0$. In other
words, this construction preserves abelianness of the hypercomplex
structure if and only if $T_{D} \, \mathfrak g$ is a trivial central
extension of $\mathfrak g$.

\begin{teo} \label{condab}Let $\{ J_{\alpha} \}$ be a hypercomplex structure on a
$4n$-dimensional  Lie algebra~$\frak g$, $D$  a flat connection such
that $D   J_{\alpha} =0$, for any $\alpha$, and $\{ {\tilde
J}_{\alpha} \}$ the induced hypercomplex structure on $T_{D} \,
{\mathfrak g}$ given by \eqref{hypercomplex}. Then $\{ {\tilde
J}_{\alpha} \}$ is abelian if and only if $\{ J_{\alpha} \}$ is
abelian on $\mathfrak g$ and $D =0$, i.e. $T_D \, {\mathfrak g} =
{\mathfrak g} \oplus \H^n$.
\end{teo}

\begin{proof}  The hypercomplex structure $\{ {\tilde J}_{\alpha} \}$ on $T_{D} \,  {\mathfrak g}$ is abelian
if and only if $\{ J_{\alpha} \}$ is abelian on $\mathfrak g$ and
$$
[ ({\tilde J}_{\alpha} (X,0), {\tilde J}_{\alpha} (0,Y)] = [ ((X,0),
(0,Y)], \quad X, Y \in {\mathfrak g}.
$$
The above equation is equivalent to the condition
\begin{equation}\label{abelian}
D_X=D_{J_{\alpha} X}  J_{\alpha},
\end{equation}
for any $\alpha = 1,2,3$ and $X \in \mathfrak g$. Let   $(\alpha ,
\beta, \gamma)$ be a cyclic permutation of $(1,2,3)$.
Equation~\eqref{abelian} implies that
$$J_{\alpha} D_{J_{\alpha} X}
=D_{J_{\alpha} X} J_{\alpha}=D_{J_{\beta} X}  J_{\beta}=J_{\beta}
D_{J_{\beta} X},
$$ for any $X\in \frak g$, therefore, $$ -D_{J_{\alpha}
X} = J_{\alpha}^2 D_{J_{\alpha} X} =J_{\alpha} J_{\beta}
D_{J_{\beta} X} = J_{\gamma} D_{J_{\gamma} (J_{\alpha} X)} =
J_{\alpha} D_{J_{\alpha}^2 X}=-J_{\alpha} D_X .$$ Hence,
$$D_X =-J_{\alpha} D_{J_{\alpha}X} = - D_{ J_{\alpha}
X}J_{\alpha},$$ which together with \eqref{abelian} implies that
$D_X=0$ for any $X\in \frak g$.
\end{proof}

\

Starting with an abelian hypercomplex structure on $\mathfrak g$ it
is possible to get a non-vanishing flat connection $D$ such that
$T_D\, \mathfrak g$  carries an HKT structure. For instance, this
can be done for the $4$-dimensional Lie algebra ${\mathfrak {aff}}
(\C)$ with its abelian hypercomplex structure (see \S\ref{4dim_ex}).
Therefore, in view of Theorem~\ref{condab}, the induced
hypercomplex structure on the tangent Lie algebra $T_D\, {\mathfrak
{aff}} (\C)$ will be non-abelian.

\

 We consider next a hyper-Hermitian Lie algebra $({\mathfrak g}, \{
J_{\alpha} \}, g)$   with a flat  connection $D$ such that $D
J_{\alpha} =0$, for any $\alpha= 1,2,3$. Let $\{ {\tilde J}_{\alpha}
\}$ be the hypercomplex structure on $T_{D} \, \frak g$ given by
\eqref{hypercomplex} and let $\tilde g$  be the inner product on
$T_{D} \, \frak g$ induced by $g$ such that $( {\mathfrak g}, 0)$
and  $(0, {\mathfrak g})$ are orthogonal. Then $\tilde g$ is
compatible with all complex structures ${\tilde J}_{\alpha}$, that
is, $(T_D\,{\mathfrak g}, \{ {\tilde J}_{\alpha} \}, \tilde g)$ is a
hyper-Hermitian Lie algebra. The following result gives a necessary
and sufficient condition for $(T_D\,{\mathfrak g}, \{ {\tilde
J}_{\alpha} \}, \tilde g)$ to be HKT.

\begin{teo} \label{construction}
Let  $({\mathfrak g}, \{ J_{\alpha} \} ,\, g)$ be a  hyper-Hermitian
Lie algebra and  $D$  a flat connection such that $D   J_{\alpha}
=0$, for any $\alpha$. Then
 $(T_{D} \, \frak g
, \{ {\tilde J}_{\alpha } \}, \tilde g)$ is  HKT if and only if $
({\mathfrak g}, \{ J_{\alpha} \} ,\, g)$ is HKT  and  $D g =0$.
\end{teo}

\begin{proof} We   first prove that the  HKT condition  is equivalent to the fact that  $ ({\mathfrak g}, \{ J_{\alpha} \} ,\, g)$ is HKT
and the operator $D_{J_{\alpha X}} J_{\alpha} - D_{J_{\beta X}} J_{\beta}$ is symmetric with respect to $g$, for any $\alpha$ and $\beta$.

By using \eqref{HKTcondition} we have that $(T_{D}  \, \frak g , \{
{\tilde J}_{\alpha } \}, \tilde g)$ is HKT if and only if
\begin{equation} \label{vanishing}
\begin{array} {l}
\tilde g  ([{\tilde J}_{\alpha} (X_1, X_2), {\tilde J}_{\alpha} (Y_1, Y_2)], (Z_1, Z_2)) + \tilde g ([{\tilde J}_{\alpha} (Y_1, Y_2), {\tilde J}_{\alpha}
Z],  (X_1, X_2))\\[4pt]
  + \tilde g  ([{\tilde J}_{\alpha} (Z_1, Z_2), {\tilde J}_{\alpha}  (X_1, X_2)], Y) - \tilde g   ([{\tilde J}_{\beta}  (X_1, X_2), {\tilde J}_{\beta} (Y_1, Y_2)], Z)\\[4pt]
   - \tilde g  ([{\tilde J}_{\beta} (Y_1, Y_2), {\tilde J}_{\beta} (Z_1, Z_2)],  (X_1, X_2))
   - \tilde g  ([{\tilde J}_{\beta} (Z_1, Z_2), {\tilde J}_{\beta} (X_1, X_2)], (Y_1, Y_2)) =0,
   \end{array}
\end{equation}
 for any $\alpha$ and $\beta$, $X_i, Y_i, Z_i \in {\mathfrak g}$, $i =1,2$. The first component  of the previous expression  vanishes if and only if $ (\{ J_{\alpha} \} ,\, g)$  is HKT on $\mathfrak g$. The  vanishing of the second component  with $X_2 =0$ and $Y_1 =0$  yields the following
 \begin{equation}\label{sym}
 g (D_{J_{\alpha X_1}} J_{\alpha} Y_2 - D_{J_{\beta X_1}} J_{\beta} Y_2 , Z_2) =
 g (Y_2, D_{J_{\alpha X_1}} J_{\alpha} Z_2 - D_{J_{\beta X_1}} J_{\beta}  Z_2),
 \end{equation}
 for any $X_1, Y_2, Z_2$. Conversely, if  $D_{J_{\alpha X}} J_{\alpha}  - D_{J_{\beta X}} J_{\beta}$ is symmetric with respect to $g$, the second component of \eqref{vanishing} is zero by direct computation.

 Therefore, from the relation
 $$
 (D_{ J_{\alpha} X} g) (J_{\alpha} Y, W) = g(D_{J_{\alpha} X} J_{\alpha} Y, W) - g(D_{J_{\alpha} X} J_{\alpha} W, Y) , \quad X, Y, W \in {\mathfrak
 g},
 $$
 we get that the operator $D_{J_{\alpha X}} J_{\alpha} - D_{J_{\beta X}} J_{\beta}$ is symmetric with respect to $g$ if and only if
 \begin{equation}
\label{covariantg}
 (D_{J_{\alpha} X} g) (J_{\alpha} Y, W) =  (D_{J_{\beta} X} g) (J_{\beta} Y, W),
\end{equation}
for any $ X, Y, W \in {\mathfrak g}$.

By \eqref{covariantg}, if $(\alpha, \beta, \gamma)$ is a cyclic
permutation of $(1,2, 3)$, we have
$$
(D_{ X} g) ( Y, W) = (D_{J_{\alpha} (J_{\alpha} X)} g) (J_{\alpha} J_{\alpha} Y, W) =  (D_{J_{\beta} (J_{\alpha} X)} g) (J_{\beta}  J_{\alpha} Y, W)   =  (D_{J_{\gamma} X} g) (J_{\gamma} Y, W).
$$
Thus
$$
g(D_X Y, W) + g(Y, D_X W)  = g(D_{J_{\gamma} X} J_{\gamma} Y, W) - g(Y, J_{\gamma} D_{J_{\gamma X}}  W),
$$
for any $X, Y, W \in {\mathfrak g}$ and $\gamma = 1,2,3$.

The above equality is equivalent to the condition  that the adjoint $(D_X - D_{J_{\gamma} X} J_{\gamma})^*$ of the operator $D_X - D_{J_{\gamma} X} J_{\gamma}$  with respect to $g$ coincides with   $-D_X - D_{J_{\gamma} X} J_{\gamma}$. Then
$$
(D_{J_{\gamma} X} J_{\gamma})^* - D_{J_{\gamma} X} J_{\gamma} = D_X + (D_X)^*.
$$
Since the left-hand side is skew-symmetric and the right one is
symmetric,  in particular  we get that $D_X = - (D_X)^*$ and thus
 the  theorem follows.

\end{proof}

In the case of the Obata connection $D=\nabla^O$ we get the
following result:
\begin{corol}
Let $({\frak g}, \{ J_{\alpha}  \}, g)$ be a hyper-Hermitian   Lie
algebra  with flat Obata connection $\nabla^O$. Then $(T_{\nabla^O}
\, {\mathfrak g}, \{  {\tilde J}_{\alpha} \}, \tilde g)$  is HKT if
and only if  $g$ is hyper-K\"ahler.
\end{corol}

\begin{proof}  Since $\nabla^O$ is  torsion free,  $\nabla^O$  is metric  if and only it coincides with the Levi-Civita connection.
\end{proof}

We will apply Theorem \ref{construction}  in the next sections in
order to construct new strong and weak  HKT manifolds.

The following result shows that the properties  of being
hyper-K\"ahler, HKT strong (resp. weak)  and balanced are preserved
by our construction.

\begin{prop}\label{properties} Let  $({\mathfrak g}, \{ J_{\alpha} \} ,\, g)$ be
an HKT Lie algebra and  $D$  a flat connection such that $Dg=0$ and
$D J_{\alpha} =0$, for any $\alpha$. Then:
\begin{enumerate}
\item[(i)] The HKT structure $( \{ {\tilde J}_{\alpha } \}, \tilde g)$ on $T_D\, \frak g$ is  hyper-K\"ahler if
and only if $ ( \{ J_{\alpha} \} ,\, g)$ is hyper-K\"ahler on $\frak
g$.

\item[(ii)]    $(\{ {\tilde J}_{\alpha} \}, \tilde g)$ is strong  (respectively  weak) if and only if
$(\{ J_{\alpha} \}, g)$ is strong (respectively weak).

\item[(iii)]   $(\{ {\tilde J}_{\alpha} \}, \tilde g)$ is balanced if and only if
$(\{ J_{\alpha} \}, g)$ is  balanced.

\end{enumerate}
\end{prop}

\begin{proof}
It follows  from \eqref{twoObata} in Lemma \ref{lem_Ob}  that the
Obata connection $\tilde \nabla^O$ corresponding to $\{ {\tilde
J}_{\alpha} \}$ satisfies $\tilde \nabla^O \tilde g =0$ if and only
if $\nabla ^O g=0$, that is, $\tilde g$ is hyper-K\"ahler if and
only if $g$ is hyper-K\"ahler.

In order to show (ii) and (iii), by  a direct computation we have
that the Bismut connection $\tilde \nabla^B$ of the new  HKT
structure $(\{ {\tilde J}_{\alpha} \}, \tilde g)$ on $T_{D} \,
{\mathfrak g} $ is related to the Bismut connection $\nabla^B$ of
the HKT structure $(\{ J_{\alpha} \},  g)$ on ${\mathfrak g}$ by
$$
 \tilde g  (\tilde \nabla^B_{(X_1, X_2)} (Y_1, Y_2), (Z_1, Z_2)) =  g(\nabla^B_{X_1} Y_1, Z_1) + g (D_{X_1} Y_2, Z_2),
$$
for any $X_i, Y_i, Z_i \in {\mathfrak g}$, $i = 1,2,3$. Therefore,
if we denote  by $\tilde c$ and $c$ the torsions of the  HKT
structures on $T_{D} \, {\mathfrak g}$ and $\mathfrak g$,
respectively, we obtain
$$
\tilde c ((X_1, X_2), (Y_1, Y_2), (Z_1, Z_2)) = c (X_1, Y_1, Z_1),
$$
and
$$
d\tilde c  \,  ((X_1, X_2), (Y_1, Y_2), (Z_1, Z_2), (W_1, W_2)) = dc
(X_1, Y_1, Z_1, W_1).
$$
This shows that the strong (respectively weak) condition is
preserved.

For (iii),  we may  use as  orthonormal basis  of $(T_{D} \,
{\mathfrak g}, \tilde g)$ the basis $\{ (e_i, 0), (0, e_i), i = 1,
\ldots, 4n \}$, where $\{e_1, \ldots, e_{4n}\}$ is an orthonormal
basis of $({\mathfrak g}, g)$.  Then, by \eqref{Leeform} we get that
the Lee form $\tilde \theta$ of the new HKT structure on $T_{D} \,
{\mathfrak g}$ is given by $\tilde \theta=\theta \circ p$, where
$\theta$ is the Lee form of the old HKT structure on $\mathfrak g$
and $p: T_D \, \frak g \to \frak g$ is the orthogonal projection.
Indeed:
$$
\begin{array}{lcl}
\tilde \theta ((X_1, X_2))  &=& -  \frac{1}{2}  \sum_{i = 1}^{4n}
{\tilde c} ( (J_1 X_1, J_1 X_2), (e_i, 0), (J_1 e_i, 0))
\\ && - \frac{1}{2} \sum_{i = 1}^{4n} \tilde c(  (J_1 X_1,  J_1X_2), (0, e_i), (0, J_1 e_i))\\[3pt]
& = &  -\frac 12  \sum_{i = 1}^{4 n} c (J_1 X_1, e_i, J_1 e_i)  =
\theta (X_1),
\end{array}
$$
for any $(X_1, X_2) \in T_D \, {\mathfrak g}$.

\end{proof}

\begin{remark}\label{higher}
The   construction of HKT structures on tangent Lie algebras given
in Theorem \ref{construction} can be iterated, since if one
considers on the HKT Lie algebra $(T_{D } \, \frak g, \{ {\tilde
J}_{\alpha} \}, \tilde g)$ the connection $\tilde D$ defined by
$$
\tilde D_{(X_1, X_2)} (Y_1, Y_2) = (D_{X_1} Y_1, D_{X_1} Y_2),
$$
$\tilde D$ is hyper-Hermitian and flat, and therefore $T_{\tilde D}
(T_{D} \, \frak g)$ admits an HKT structure (compare with
\cite{BD}). Moreover, starting with a balanced $4n$-dimensional HKT
Lie algebra $({\mathfrak g}, \{ J_{\alpha} \} ,\, g)$ (see, for
instance, \S\ref{new_ex}), Proposition~\ref{properties} implies that
the successive tangent HKT Lie algebras are balanced. Since balanced
HKT structures have holonomy in $SL(m,\Bbb H )$, then  the HKT
structures on the
 simply connected Lie groups corresponding to the  tangent Lie algebras and on any of their
compact quotients (provided these exist) have holonomy in $SL(m,\Bbb
H )$ for some $m\geq n$.
\end{remark}

\

Let $({\mathfrak g}, \{ J_{\alpha} \} ,\, g)$ be a hyper-Hermitian
$4n$-dimensional Lie algebra and  $D$~a~flat connection on $\frak
g$. The following conditions are equivalent:
\begin{enumerate}
 \item $Dg =0$ and $D J_{\alpha}=0, \; \alpha =1,2,3$;
\item
    for any $X
\in \frak g$,  $D_X $ is skew-symmetric with respect to $g$ and
commutes with $J_{\alpha}$ for each $\alpha$;
 \item the
map $D : \frak g \to \frak{sp} (n)$ is a Lie algebra homomorphism.
\end{enumerate} Therefore, Theorem~\ref{construction} can be
rephrased as follows:

\begin{corol} \label{homomorphism} Let  $({\mathfrak g}, \{ J_{\alpha} \} ,\, g)$
be a  hyper-Hermitian $4n$-dimensional Lie algebra and  $D$~a~flat
connection such that $D   J_{\alpha} =0$, for any $\alpha$. The
hyper-Hermitian structure $(\{ {\tilde J}_{\alpha} \}, \tilde g)$ on
the  tangent algebra  $T_{D} \,  \frak g$  is HKT  if and only if
$({\mathfrak g},\{ J_{\alpha} \}, g)$ is HKT  and $D: \frak g \to
\frak{sp} (n)$ is a Lie algebra homomorphism.
\end{corol}

\

\begin{remark} If $G$ is a simply connected Lie group with Lie algebra $\frak g$, there is a one-to-one correspondence
between Lie algebra homomorphisms $D: \frak g \to \frak{sp} (n)$ and
quaternionic unitary representation of $G$ on $\H^n$.\end{remark}

\

The previous construction can be generalised by replacing the flat
connection $D$ on $\mathfrak g$ by a quaternionic representation of
$\mathfrak g$ on $\Bbb H^q$. In fact,
  given a $4n$-dimensional    Lie algebra ${\mathfrak g}$ and a Lie algebra  homomorphism
$$
\rho:  {\mathfrak g} \to {\mathfrak {gl}} (q, \H),
$$
 instead of considering  the tangent Lie algebra  of $\mathfrak g$, one can define on $T_{\rho} \, {\mathfrak g} : = \mathfrak g \ltimes _{\rho} \H^q$ the
 following Lie bracket:
 \begin{equation}\label{t_rho}
 [(X, V), (Y, W)] = ([X, Y], \rho(X) W - \rho(Y) V),
 \end{equation}
 for  any $X, Y \in \mathfrak g$ and $V, W \in \H^q$.

   Given a hypercomplex structure $\{J_{\alpha} \}$ on $\mathfrak g$ we define $\{ \tilde J_{\alpha} \}$ on  $T_{\rho} \, {\mathfrak g}$
 as follows:
 \begin{equation}\label{hcx}
 \tilde J_{\alpha} (X, V) = (J_{\alpha} X, i_{\alpha} V),
 \end{equation}
 with  $i_ 1 = i, \, i_2 =j,\, i_3 = k$, where  $i_{\alpha} V$ denotes  left quaternionic multiplication by $i_{\alpha}$  on $\H^q$.
 The integrability of $\tilde J_{\alpha}$ follows from \cite[Proposition~3.3]{BD}.

 If  $q = n$,  the  Lie algebra $T_{\rho} \, \mathfrak g$  coincides  with the tangent Lie algebra $T_D \mathfrak g$
 associated to the connection $D = \rho$.

 If $(\mathfrak g ,\{ J_{\alpha} \} ,\, g)$ is hyper-Hermitian, the
Lie algebra $T_{\rho} \, {\mathfrak g}$ becomes hyper-Hermitian with
the hypercomplex structure given by \eqref{hcx} and  the inner
product $\tilde g$ induced by $g$ and the natural inner product on
$\H^q$ such that $\frak g$ is orthogonal to $\H^q$.

\

\begin{remark}\label{cor_t_rho} We point out that Corollary \ref{homomorphism}  goes through with the obvious changes for the Lie algebra $T_{\rho}\,
\frak g$, and it turns out that the HKT stucture $(\{ {\tilde
J}_{\alpha} \}, \tilde g)$ on $T_{\rho}\, \frak g$ is strong
(respectively, weak, hyper-K\"ahler, balanced) if and only if $(\{ {
J}_{\alpha} \},  g)$ is strong (respectively, weak, hyper-K\"ahler,
balanced). The analogue of
 Theorem~\ref{condab} is also true for
$T_{\rho}\, \frak g$. In fact, given a hypercomplex structure $\{
J_{\alpha} \}$ on $\frak g$, consider the induced hypercomplex
structure $\{ {\tilde J}_{\alpha} \}$ on $T_{\rho}\, \frak g$ as in
\eqref{hcx}. An analogous proof to that of Theorem~\ref{condab}
gives that
 $\{ {\tilde J}_{\alpha} \}$  is
abelian  on $T_{\rho}\, \frak g$ if and only if $\{ J_{\alpha} \}$
is abelian on $\mathfrak g$ and $\rho =0$, i.e. $T_{\rho} \,
{\mathfrak g} = {\mathfrak g} \oplus \H^q$.

\end{remark}

\

 Let $G$ be a  connected Lie group with Lie algebra $\mathfrak g$
 and   $\pi: G \to GL (q, \H)$
  a  quaternionic representation of $G$ on $\Bbb H^q$.
 The following result
is a
 generalisation of Theorem \ref{construction} and its proof  follows by using analogous arguments, replacing
    the connection $D$ with the representation $\rho=(d\pi)_e$ of $\mathfrak
    g$. This may be
regarded as the left-invariant counterpart of \cite[Theorem
7.2]{Ve2}, where it has been shown that the natural metric on the
total space of a hyperholomorphic bundle over an HKT manifold is
HKT.

 \begin{teo} \label{generalization}
Let  $(\{ J_{\alpha} \} ,\, g)$  be a left invariant hyper-Hermitian
structure on a connected Lie group~$G$ and  $\pi : G \to GL (q, \H)$
 a quaternionic representation. Then $(T_{\pi} \, G , \{
{\tilde J}_{\alpha } \}, \tilde g)$ is  HKT if and only if $ (G, \{
J_{\alpha} \} ,\, g)$ is HKT  and  $\pi$ is unitary, that is, $\pi
(G)$ is contained in $Sp (q)$.

Moreover,
 the new HKT structure  $(\{ {\tilde J}_{\alpha} \}, \tilde g)$ is strong  (respectively  weak, hyper-K\"ahler, balanced) if and only if
$(\{ J_{\alpha} \}, g)$ is strong (respectively weak,
hyper-K\"ahler, balanced).
 \end{teo}

\

We show next that Theorem~\ref{generalization} is a useful tool for
constructing new examples of compact HKT manifolds with holonomy in
$SL(n, \Bbb H)$. In fact, we obtain as an application of the results
in this section  a new compact HKT manifold which is balanced,
therefore its holonomy is contained in $SL(3, \Bbb H )$
\cite[Remark~4.9]{Ve4}. We point out that this manifold is the first
known example of a compact
$12$-dimensional HKT manifold with holonomy in $SL(n, \Bbb H)$ which
is not a nilmanifold (see \cite[Remark~4.4]{Ve4}).

\subsection{A balanced HKT solvmanifold.} \label{new_ex}
Let $(\frak g, \{ J_{\alpha} \} ,\, g)$ be the $8$-dimensional HKT
Lie algebra with basis $\{ e_1, \dots , e_8 \} $ and dual basis $\{
e^1, \dots , e^8 \} $ of $\frak g^*$, such that the Lie bracket is
given as follows:
\begin{align*}[e_5,e_k]&=e_{k-4},\;
k=6,7,8,&  [e_6,e_7]&= -e_4 \\
[e_6,e_8]&= e_3, &   [e_7,e_8]&= -e_2,
\end{align*}
with inner product \[ g=\sum_{i=1}^8 (e^i)^2,  \] and hypercomplex
structure
\begin{align*} J_1e_1&=e_2,& \quad J_1e_3 &=e_4, &\quad J_1e_5&=e_6, &\quad J_1e_7& =e_8,
&\quad J_1 ^2 &=-\text{id}, \\
J_2e_1&=e_3,& \quad J_2e_2 &=-e_4, &\quad J_2e_5&=e_7, &\quad
J_2e_6& =-e_8,
&\quad J_2 ^2 &=-\text{id}, \\
J_3e_1&=e_4,& \quad J_3e_2 &=e_3, &\quad J_3e_5&=e_8, &\quad J_3e_6&
=e_7, &\quad J_3 ^2 &=-\text{id}.
\end{align*}
We observe that span $\{e_2, \dots , e_8\}$ with the inner product
induced by $g$ is a $2$-step nilpotent Lie algebra of Heisenberg
type (see \cite{K}) and $\{ J_{\alpha}\}$ is an abelian hypercomplex
structure on $\frak g$. Since the inner product $g$ is
hyper-Hermitian, it follows from \cite{DF} that $(\frak g, \{
J_{\alpha} \} ,\, g)$ is an HKT Lie algebra. Moreover,
\cite[Proposition~4.11]{BDV} implies that $g$ is balanced. We fix a
basis $\{ f_1, \dots , f_4\}$ of $\Bbb R ^4$ and consider the
endomorphism $\rho : \frak g \to \text{End}(\Bbb R^4)$ defined by:
\[ \rho(e_1)=\begin{pmatrix} 0&-1& & \\
1&0& & \\
 & &0&1\\
 & &-1&0
\end{pmatrix}, \qquad \rho (e_k)=0, \; k=2, \dots , 8.
 \]
We extend $\{J_{\alpha}\}$ to a hypercomplex structure $\{\tilde
J_{\alpha}\}$ on $T_{\rho}\, \frak g$ as follows:
\begin{align*} \tilde J_1f_1&=f_2,& \quad \tilde J_1f_3 &=f_4,
&\quad \tilde J_1 ^2 &=-\text{id}, \\
\tilde J_2f_1&=f_3,& \quad\tilde J_2f_2 &=-f_4, &\quad  \tilde J_2 ^2 &=-\text{id}, \\
\tilde J_3f_1&=f_4,& \quad \tilde J_3f_2 &=f_3, &\quad  \tilde J_3
^2 &=-\text{id}.
\end{align*}
 Let $\tilde g$ be the inner product on $T_{\rho}\, \frak g$ obtained by extending  $g$ in the obvious way. It follows that $\rho: \frak g\to
\frak{sp}(1)$ is a Lie algebra homomorphism, therefore, $(T_{\rho}\,
\frak g ,\{\tilde J_{\alpha}\}, \tilde g)$ is a $12$-dimensional
balanced HKT Lie algebra (see Remark~\ref{cor_t_rho}). Let
$S=N\ltimes \Bbb R ^4$ be the simply connected solvable Lie group
with Lie algebra $T_{\rho}\, \frak g$, where $N$ is the simply
connected $2$-step nilpotent Lie group with Lie algebra $\frak g$.
It is well known that $N$ admits a lattice $\Gamma _1$ \cite{K, R},
which can be chosen to be compatible with $\{ J_{\alpha}\}$, and
there exists a lattice $\Gamma_2$ in $\Bbb R ^4$ such that the
action of $N$ on $\Bbb R ^4$ is compatible with $\Gamma_2$.
Therefore, setting $\Gamma:= \Gamma_1 \ltimes \Gamma_2$, $\Gamma$ is
a lattice in $S$ and $\Gamma\backslash S$ is a solvmanifold carrying
a balanced HKT structure, hence it has holonomy in $SL(3, \Bbb H)$.
This manifold is an $S^1$-bundle over a $2$-step nilmanifold. We
point out that  $\{\tilde J_{\alpha}\}$ is not abelian on
$T_{\rho}\, \frak g$, therefore this  provides a counterexample to
the analogue of \cite[Theorem~4.6]{BDV} for HKT solvmanifolds, since
$\tilde g$ is HKT but the corresponding hypercomplex structure is
not abelian.

\

\section{ Weak HKT structures  on $8$-dimensional tangent Lie
algebras}\label{4dim_ex}

In this section we  apply  Corollary~\ref{homomorphism}  to all
$4$-dimensional non-reductuve Lie algebras $\mathfrak g$ admitting
hypercomplex structures to obtain HKT structures on the
corresponding tangent Lie algebras $T_D\, \frak  g$. We will proceed
as follows. Given an HKT structure $(\{ J_{\alpha}\},g)$ on $\frak
g$, we will consider three endomorphisms $\{J_1',J_2',J_3'\}$ which
form a basis of $\frak{sp}(1)$ and we will provide an explicit
homomorphism $D:\frak g \to \frak{sp}(1)$ by expressing $D_X$ as a
linear combination of $\{ J_{\alpha}'\}$ for any $X$ in a basis of
$\frak g$.

Given a $4$-dimensional Lie algebra $\frak g$ with basis $\{ e_1,
\dots , e_4 \} $, let $\{ e^1, \dots , e^4 \} $ be the basis of
$\frak g^*$ dual to $\{ e_1, \dots , e_4 \} $. By \cite{Ba} a
non-abelian $4$-dimensional real Lie algebra admitting a
hypercomplex structure is isomorphic to one of the following Lie
algebras:
\begin{enumerate}
\item ${\mathfrak {sp}} (1) \oplus  {\mathfrak u} (1)$;
\item ${\mathfrak {aff}} (\C) = (-e^{13} + e^{24}, - e^{23} - e^{14}, 0, 0)$;
\item $(0, - e^{12}, - e^{13}, - e^{14})$;
\item $(0, - \frac{1}{2} e^{12}, - \frac{1}{2} e^{13}, - e^{23} - e^{14})$,
\end{enumerate}
where   for instance $(-e^{13} + e^{24}, - e^{23} - e^{14}, 0, 0)$ denotes the Lie algebra with non-zero Lie backets
$$
\begin{array}{l}
[e_1, e_3] = e_1  = - [e_2, e_4],\\[4pt]
 [e_2, e_3] = e_2 = [e_1, e_4].
\end{array}
$$

The Lie algebra ${\mathfrak {sp}} (1) \oplus  {\mathfrak u} (1)$ is
the only reductive one in the above list and admits by \cite{GP} a
strong HKT structure, that we will consider in the next section in
order to construct new strong HKT examples.

 The Lie algebra $(2)$ is the only $4$-dimensional one admitting an abelian
hypercomplex structure. Moreover, any hypercomplex structure on
${\mathfrak {aff}} (\C)$ is abelian \cite{Ba}. Take, for instance:
\begin{equation} \label{hyperstruct}
\begin{array}{ll}
J_1 e_1 = - e_4, &\, J_1 e_2 = e_3,\\[4pt]
J_2 e_1 = e_2, &\, J_2 e_3 = - e_4 .
\end{array}
\end{equation}
If we consider the inner product
\begin{equation} \label{metric}
g = \sum_{i = 1}^4  (e^i)^2,
\end{equation}
 it turns out that $(\{ { J}_{\alpha} \},  g)$ defines a weak HKT structure on
${\mathfrak {aff}} (\C)$. There is a Lie algebra  homomorphism
$D:{\mathfrak {aff}} (\C) \to {\mathfrak {sp}} (1)$ defined as
follows
$$
\begin{array}{l}
D_{e_1} = D_{e_2 } =0,\\[4pt]
D_{e_3} = a_1 J'_1 + a_2 J'_2 + a_3 J'_3, \qquad D_{e_4} =
bD_{e_3},
\end{array}
$$
where $a_i, i =1,2,3$, $b$ are real numbers and the endomorphisms
$J'_{\alpha}$ are given by
\begin{equation} \label{endom}
\begin{array}{lll}
J_1' e_1 = e_4, \, &J_1' e_2 = e_3, \, &{J}_1'^2 = - {\mbox {id}},\\[4pt]
J_2' e_1 = -e_2, \, &J_2' e_3 = - e_4, \, &J_2'^2 = -{\mbox {id}},\\[4pt]
J_3'  e_1 = -e_3, \,& J_3'  e_2= e_4, \,, &J_3'^2 = -{\mbox {id}}.
\end{array}
\end{equation}
Corollary~\ref{homomorphism} implies that the induced
hyper-Hermitian structure $(\{ {\tilde J}_{\alpha} \}, \tilde g)$ is
weak HKT  on $T_{D} \, {\mathfrak {aff}} (\C)$ and the hypercomplex
structure $\{ {\tilde J}_{\alpha} \}$ is not abelian (see
Theorem~\ref{condab}).

For the Lie algebra $(3)$   we can consider the  weak HKT  structure
$(\{ J_{\alpha} \}, g)$  defined by \eqref{hyperstruct} and
\eqref{metric}. In this case, we have the following   homomorphism
$D$:
\begin{equation} \label{exprj'}
\begin{array}{l}
D_{e_1} = a_1 J'_1 + a_2 J'_2 + a_3 J'_3,\\[4pt]
D_{e_2} = D_{e_3}  = D_{e_4} =  0,
\end{array}
\end{equation}
where  $a_i, i =1,2,3$ are real numbers and  the endomorphisms
$J'_{\alpha}$ are as in \eqref{endom}.

For  the last Lie algebra $(4)$  one takes the hyper-Hermitian structure  $(\{ J_{\alpha} \}, g)$ with
$$
\begin{array}{ll}
J_1 e_1 = e_4, &\, J_1 e_2 = -e_3,\\[4pt]
J_2 e_1 =  \frac {\sqrt{2}} {2} \, e_2, \quad  & \, J_2 e_4 =  \frac
{\sqrt{2}} {2} \, e_3,
\end{array}
$$
and $g$ given by
\[  g=   (e^1)^2 +(e^4)^2 + 2\left((e^2)^2 +(e^3)^2\right).\]

 Let  $D$ be the homomorphism defined in
\eqref{exprj'} with  the following endomorphisms $J'_{\alpha}\,$:
$$
\begin{array}{lll}
J_1' \,  e_1 = - e_4, \, &J_1'  e_2 =-  e_3, \, &J_1'^2 = - {\mbox {id}}\\[4pt]
J_2' e_1 = - \frac {\sqrt{2}} {2} \, e_2,, \, &J_2'  e_4 =\frac {\sqrt{2}} {2} \, e_3, \, &J_2'^2 = -{\mbox {id}},\\[4pt]
J_3' e_1 =\frac {\sqrt{2}} {2} \, e_3, \,& J_3'  e_4= \frac {\sqrt{2}} {2} \, e_2, \,, &J_3'^2 = -{\mbox {id}}.
\end{array}
$$
The induced hyper-Hermitian structure  $(\{ {\tilde J}_{\alpha} \},
\tilde g )$ on the tangent Lie algebra is weak HKT. We point out
that the Obata connection $\nabla ^O$ corresponding to $(\frak g ,
\{ J_{\alpha} \})$ has holonomy $\frak{sl}(1, \Bbb H )$, therefore,
$\frak{hol} (\tilde \nabla ^O)=\frak{sl}(1, \Bbb H )$ (Lemma
\ref{lem_Ob}). In fact, the Obata connection $\nabla^O$ has been
calculated  in \cite{Ba}:
\[  \nabla^O_{e_1}  =\frac34 \, \text{id}, \quad \nabla^O_{e_2}
=-\frac{\sqrt{2}}4 \, J_2', \quad \nabla^O_{e_3} =\frac{\sqrt{2}}4
\,  J_3',\quad \nabla^O_{e_4} =\frac{1}4 \, J_1',
\]
hence, $\frak{hol} ( \nabla ^O)=\text{span} \{ J_{\alpha}' \}\cong
\frak{sl}(1, \Bbb H )$. In particular, the canonical bundles of the
corresponding simply connected Lie groups (with respect to
$J_{\alpha}$ and $\tilde J_{\alpha}$, respectively, for any
$\alpha$) are holomorphically trivial (see \cite{Ve}). However, we
observe that $g$ (hence $\tilde g$) is not balanced.


 \

\section{New    compact  strong HKT manifolds  }

In \cite{Jo} Joyce showed that  for  any compact semisimple  Lie
algebra ${\mathfrak s}$   there exists a non-negative  integer $k$
such that ${\mathfrak s} \oplus k \, {\mathfrak {u}} (1)$ admits a
hypercomplex structure $\{ J_{\alpha} \}$, constructed by using the
decomposition of  ${\mathfrak s}$ in terms of certain  ${\mathfrak
{sp}}(1)$ subalgebras. By \cite{OP, GP}  there exists a compatible
hyper-Hermitian inner product  $g$ such that its restriction    to
$\mathfrak s$ is equal to  the opposite of the Killing-Cartan form.
Moreover,  the hyper-Hermitian structure $(\{ J_{\alpha} \}, g)$ is
strong HKT and induces  a left-invariant HKT structure  on each Lie
group  $G$ with Lie algebra ${\mathfrak s} \oplus k \, {\mathfrak
{u}} (1)$.

In this section we  apply Theorem~\ref{construction} to obtain HKT
structures on non-trivial extensions
$$T_{D} ( {\frak s} \oplus k \,{\mathfrak u} (1)) =  \left({\frak s}
\oplus k \, {\mathfrak u} (1) \right)\ltimes _D \H^{n},$$ where
 $ 4n =\dim {\mathfrak s}  + k$ and $D$ is an appropriate flat hyper-Hermitian  connection on ${\frak s} \oplus  k \, {\mathfrak u} (1)$.
For instance, one could consider  $D=\nabla^ O$, the Obata
connection, but $\nabla^O$ is not hyper-Hermitian, since this would
imply that $g$ is hyper-K\"ahler on ${\frak s} \oplus  k \,
{\mathfrak u} (1)$, which is impossible. If we consider
$D=\nabla^B$,  the Bismut connection  associated to  $(\{ J_{\alpha}
\}, g)$,  $D$ is hyper-Hermitian and flat, but in this case it is
identically zero and  therefore  $T_{\nabla^B} ( {\frak s} \oplus  k
\,{\mathfrak u} (1))$ is  a trivial central extension of ${\mathfrak
s} \oplus k \, {\mathfrak {u}} (1)$.


In order  to get non-trivial extensions we will investigate the
existence of other flat hyper-Hermitian  connections on ${\mathfrak
s} \oplus k \, {\mathfrak {u}} (1)$. The simplest case is the
$4$-dimensional  reductive Lie algebra ${\mathfrak {sp}}  (1) \oplus
{\mathfrak u} (1)$.

 Starting with  the  standard  strong HKT structure  $(\{ J_{\alpha} \}, g)$  on  the
 Lie algebra  ${\mathfrak {sp}}  (1) \oplus {\mathfrak u} (1)$ considered in \cite{GP}, we will show that ${\mathfrak {sp}}  (1) \oplus {\mathfrak u} (1)$ admits a flat hyper-Hermitian connection $D$, essentially defined by the projection homomorphism from ${\mathfrak {sp}}  (1) \oplus {\mathfrak u} (1)$ to ${\mathfrak {sp}} (1)$.

 Let $\{ e_1, e_2, e_3
 , e_4 \}$  be the  standard  basis of  ${\mathfrak {sp}}  (1) \oplus {\mathfrak u} (1)$ with non zero Lie brackets
 $$
 [e_1, e_2] = e_3, \quad  [e_2, e_3] = e_1, \quad  [e_3, e_1] = e_2.
 $$
 Consider the hyper-Hermitian  structure given by the hypercomplex structure
 $$
 \begin{array} {l}
 J_1 e_1 = - e_4, \quad J_1 e_2 = e_3,\\
J_2 e_1 = - e_3, \quad J_2 e_2 = -e_4
\end{array}
$$
and by the inner product $g$ such that the basis $\{ e_1, \ldots,
e_4 \}$  is orthonormal. It follows by  \cite{GP} that $(\{
J_{\alpha} \}, g)$ is strong HKT.

By  \cite{Ba} the centralizer of $\{ J_1, J_2 \}$ in ${\mbox{End}}
({\mathfrak {sp}}  (1) \oplus {\mathfrak u} (1))$, which is
isomorphic to ${\mathfrak {gl}} (1, \H)$, is  spanned by  the
identity ${\mbox {id}}$  and by  the following endomorphisms:
$$
\begin{array}{lll}
J'_1 e_1 =  e_4, &\, J'_1 e_2 = e_3, &\, {J'_1}^2 = - {\mbox {id}},\\
J'_2 e_1 =  -e_3, &\, J'_2 e_2 = e_4, &\, {J'_1}^2 = - {\mbox {id}},\\
J'_3 e_1 =  e_2, &\, J'_3 e_3 = e_4, &\, {J'_3}^2 = - {\mbox {id}}.
\end{array}
$$
Therefore, the connection $D$ on ${\mathfrak {sp}}  (1) \oplus
{\mathfrak u} (1)$ defined by
$$
D_{e_1} =  \frac 12 J'_1,  \quad D_{e_2} =  \frac 12 J'_2, \quad
D_{e_3} = \frac 12 J'_3, \quad D_{e_4} =  0,
$$
is  flat and hyper-Hermitian with respect  to $g$.

In view of Corollary \ref{homomorphism}, the connection $D$
corresponds to the projection
\begin{equation} \label{sp1}
{\frak {sp} }(1) \oplus \R \, e_4 \rightarrow {\frak {sp} }(1).
\end{equation}

In this way  we can apply  Theorem \ref{construction}  in order to get a new $8$-dimensional Lie algebra $T_{D}  \,({\mathfrak {sp}}  (1) \oplus {\mathfrak u} (1))$ with structure equations
$$
\begin{array}{l}
[e_1, e_2] = e_3, \quad  [e_2, e_3] = e_1, \quad  [e_3, e_1] = e_2,\\[4pt]
{[e_1, e_8] =  - [e_2, e_7] =   [e_3, e_6]  = - \frac 12 e_5},\\[4pt]
  { [e_1, e_7] =   [e_2, e_8] =  - [e_3, e_5]  = - \frac 12 e_6},\\[4pt]
  {[e_1, e_6] =  - [e_2, e_5] =  - [e_3, e_8]  = \frac 12 e_7}, \\[4pt]
{[e_1, e_5] =   [e_2, e_6] = - [e_3, e_7]  = \frac 12 e_8}.
\end{array}
$$
The induced  hypercomplex structure $\{ {\tilde J}_{\alpha} \}$  on
$T_{D} \, ({\mathfrak {sp}}  (1) \oplus {\mathfrak u} (1))$    (see
\eqref{hypercomplex})  and the inner product $\tilde g$ give a
left-invariant strong HKT structure  $(\{ {\tilde J}^+_{\alpha} \},
\tilde g)$  on the  simply connected Lie group with Lie algebra
$T_{D} \, {\mathfrak g}$. Such a Lie group is a product of $\R$ by
the $7$-dimensional Lie group $SU(2) \ltimes \R^4$. This
$7$-dimensional Lie group was already considered in \cite{FT}, where
it was shown that it admits a compact quotient $M^7$. Therefore, we
obtain a compact $8$-dimensional manifold $M^7 \times S^1$ admitting
a strong HKT structure with flat Obata  connection.
%

\begin{remark} In view of Remark~\ref{higher}, the previous construction can be iterated in higher dimensions and
this allows to find new compact examples of dimension $ 2^{n+2}$, for any $n \geq 1$.
\end{remark}

More in general, we can  consider the compact semisimple Lie algebra
${\mathfrak {sp}}(l)$. By   \cite{Jo, SSTV, GP}  it follows that
there exists   a strong HKT structure $(\{ J_{\alpha}  \} , g)$ on
the direct sum   ${\mathfrak {sp}} (l) \oplus  l \, {\mathfrak u}
(1)$.

 For any $l \geq 1$, it is  possible to construct a homomorphism
$$
{\mathfrak {sp}} (l) \oplus     l  \, {\mathfrak {u}} (1)  \rightarrow {\mathfrak {sp}} \left(\frac{l  (l + 1)}{2} \right)
$$
given  by
$$
(X, Y) \in   {\mathfrak {sp}} (l) \oplus  l  \, {\mathfrak {u}} (1)  \mapsto i(X),
$$
where the inclusion   $i:  {\mathfrak {sp}} (l) \hookrightarrow
{\mathfrak {sp}} \left(\frac{l (l + 1)}{2} \right)$ is a natural
injective homomorphism.  This generalises the  case $l=1$
considered previously (see \eqref{sp1}), where $i$ was the
identity of $ {\mathfrak {sp}}  (1)$.

Therefore, applying   Corollary  \ref{homomorphism},  we have the
following:

\begin{prop} For any $l \geq 1$, the  Lie algebra    ${\mathfrak {sp}} (l) \oplus l \,  {\mathfrak u} (1)$  with the strong  HKT structure
$(\{ J_{\alpha} \}, g)$  admits a flat hyper-Hermitian connection
$D$ and   the induced hyper-Hermitian structure on the  tangent
algebra $T_{D} ( {\mathfrak {sp}} (l) \oplus l \,  {\mathfrak u}
(1))$ is strong HKT.
\end{prop}

\

 Consider  the following  strong HKT
reductive Lie algebras $\mathfrak g$ (see \cite{SSTV}): \begin{equation} \label{list}
\begin{array}{c}
{\mathfrak {su}}(2 l + 1), \; (l>1),\qquad {\mathfrak {su}}(2 l)
\oplus {\mathfrak u} (1), \qquad
{\mathfrak {so}}(2 l + 1) \oplus l\, {\mathfrak u} (1), \; (l>3), \\[4pt]
   {\mathfrak {so}}(4 l ) \oplus  2 l \, {\mathfrak u} (1), \quad\quad {\mathfrak {so}}(4 l + 2) \oplus (2 l -1) \, {\mathfrak u} (1).
\end{array}
\end{equation}
 It is possible to find a Lie algebra
homomorphism  $D$ between ${\mathfrak g}$ and ${\mathfrak {sp}}
(n)$, where $4 n = \dim  \, \mathfrak g$,  given by an  inclusion of
the Levi subalgebra of $\mathfrak g$ in  ${\mathfrak {sp}} (n)$.
Therefore, applying Corollary  \ref{homomorphism},  we  get new
strong HKT structures on $T_{D} \,  {\mathfrak g}$. Note that the
Lie algebra $T_D \,  \mathfrak g$ has a Levi-decomposition with a
compact Levi subalgebra  and  an abelian radical.   Since $D$ is the
inclusion,  for any   Lie algebra $\mathfrak g$  in the list
\eqref{list},  there exists a connected Lie group $K$ with Lie
algebra  $T_D \, \mathfrak g$  such that the Levi factor of $K$ is
compact. We observe that $K$ is a subgroup of the isometry group
$E(4n)$ of the Euclidean $4n$-dimensional space. Therefore, the
existence of lattices, which in this case are crystallographic
groups, is guaranteed by Bieberbach's theorem \cite{Bieb}. Moreover,
if one fixes~$n$, $E(4n)$  has only a finite number of lattices up
to isomorphism, hence, so does $K$. In conclusion we have the
following:

\begin{teo} \label{compactHKT}
Let $K$ be a  connected  Lie subgroup    of  $E(4n)$  with Lie
algebra $T_D \,  \mathfrak g$, where $\mathfrak g$ is isomorphic to
one of the Lie algebras in the list \eqref{list}. Then $K$ admits a
left-invariant strong HKT structure which induces a strong HKT
structure on any compact quotient of~$K$.
\end{teo}

\smallskip

\begin{remark}\label{rem_su(3)} When $\frak g = {\mathfrak {so}}(7) \oplus 3 \, {\mathfrak u} (1)$ or
${\mathfrak {su}} (3)$, there exists no Lie algebra homomorphism $D:
{\mathfrak {su}} (3)\to {\mathfrak {sp}} (2)$ or $D: {\mathfrak
{so}}(7) \oplus 3 \, {\mathfrak u} (1) \to {\mathfrak {sp}} (6)$.
This follows by computing the dimension of the irreducible
fundamental representations of $SU(3)$ and $SO(7)$. However, in both
cases there exists a positive integer $q$ and a Lie algebra
homomorphism $\rho : \frak g \to {\mathfrak {sp}} (q)$. Applying
Theorem \ref{generalization}, we obtain  new strong HKT Lie algebras
$T_{\rho}\, \frak g$.
\end{remark}




\

\section{HKT  structures  associated to  Hermitian Lie algebras}

Starting with  a Hermitian Lie algebra $({\mathfrak g}, J,\, g)$
with a   flat torsion-free connection $D$ which is complex, i.e.  such that $D J =0$, we can
construct a hyper-Hermitian  Lie algebra.  The previous  torsion-free connection is then {\em complex-flat} in the terminology of \cite{Jo2}.  Next, we prove  that the new
 hyper-Hermitian Lie algebra is HKT if and only if $({\mathfrak g}, J,\, g)$ is
K\"ahler.

Consider  as in Section \ref{Section} the tangent Lie algebra
$T_{ D} \, \frak g := {\mathfrak g} \ltimes_{D} {\mathfrak
g}$ with the following Lie bracket:
$$  [(X_1,X_2), (Y_1,Y_2)]=([X_1, Y_1] ,D_{X_1} Y_2 -  D
_{Y_1} X_2 )
$$
and the following hypercomplex structure (see \cite[Corollary
4.3]{BD}):
\begin{equation}  \label{hk}
J_1(X_1, X_2)= (JX_1, - JX_2), \quad  J_2(X_1, X_2)= (X_2, -X_1),
\quad  J_3 = J_1J_2.
\end{equation}
If one considers the inner product  $\tilde g$ induced by $g$ such
that $( {\mathfrak g}, 0)$ is orthogonal to $(0, {\mathfrak g})$,
the following theorem shows that   $(\{ J_{\alpha} \}, \tilde g)$ is
HKT on $T_D \, \frak g$ if and only if $( J,\, g)$ is K\"ahler on
$\frak g$.
\begin{teo} \label{construction2} Let $({\mathfrak g}, J,\,  g)$ be a Hermitian Lie algebra and  $D$
  a flat torsion-free
complex connection on ${\mathfrak g}$. Then
 $(T_{D} \,  {\mathfrak g}, \{ J_{\alpha } \},
\tilde g)$ is HKT if and only if
   $({\mathfrak g}, J, g)$  is K\"ahler.
\end{teo}

\begin{proof}
By \cite{DF} $(T _{D} {\mathfrak g}, \{ J_{\alpha }
\}, \tilde g)$ is HKT  if and only if
$$
\begin{array} {l}
\tilde g ([J_1 (X_1, X_2), J_1 (Y_1, Y_2)], (Z_1, Z_2)) +\tilde g ([J_1 (Y_1,
Y_2),
J_1 (Z_1, Z_2)], (X_1, X_2))\\[4pt]
  + \tilde g ([J_1 (Z_1, Z_2), J_1 (X_1, X_2)], (Y_1, Y_2)) =\tilde g ([J_3 (X_1,
X_2), J_3 (Y_1, Y_2)], (Z_1, Z_2))\\[4pt]
   + \tilde g ([J_3 (Y_1, Y_2), J_3 (Z_1, Z_2)], (X_1, X_2))
   + \tilde g  ([J_3 (Z_1, Z_2), J_3 (X_1, X_2)], (Y_1, Y_2)),
   \end{array}
$$
for any $X_i, Y_i, Z_i \in {\mathfrak g}$, $i = 1, 2$. One has
\begin{equation}\label{eq1}
\begin{array}{c}
\tilde g  ([J_1 (X_1, X_2), J_1 (Y_1, Y_2)], (Z_1, Z_2)) +\tilde g ([J_1 (Y_1,
Y_2),
J_1 (Z_1, Z_2)], (X_1, X_2))\\[4pt]
  + \tilde g  ([J_1 (Z_1, Z_2), J_1 (X_1, X_2)], (Y_1, Y_2)) = g ([JX_1, J
Y_1], Z_1) + g ([JY_1, J Z_1], X_1)\\[4pt]
   + g ([JZ_1, J X_1], Y_1) - g (D_{JX_1} J Y_2, Z_2) - g
(D_{JY_1} J Z_2, X_2) - g (D_{JZ_1} J X_2,
Y_2)\\[4pt]
    + g (D_{JY_1} J X_2, Z_2) + g (D_{JZ_1} J
Y_2, X_2) + g (D_{JX_1} J Z_2, Y_2).
  \end{array}
  \end{equation}
On the other hand,
\begin{equation}\label{equation2}
\begin{array}{c}
\tilde g ([J_3 (X_1, X_2), J_3 (Y_1, Y_2)], (Z_1, Z_2)) +\tilde g ([J_3 (Y_1,
Y_2),
J_3 (Z_1, Z_2)], (X_1, X_2))\\[4pt]
  + \tilde g  ([J_3 (Z_1, Z_2), J_3 (X_1, X_2)], (Y_1, Y_2)) = g ([JX_2, J
Y_2], Z_1) + g ([JY_2, J Z_2], X_1)\\[4pt]
   + g ([JZ_2, J X_2], Y_1) - g (D_{JX_2} J Z_1, Y_2) - g
(D_{JY_2} J X_1, Z_2) - g (D_{JZ_2} J Y_1,
X_2)\\[4pt]
    + g (D_{JY_2} J Z_1, X_2) + g (D_{JZ_2} J
X_1, Y_2) + g (D_{JX_2} J Y_1, Z_2).
  \end{array}
  \end{equation}
Substracting the right-hand sides of the equations  \eqref{eq1}
and \eqref{equation2}, and  using the fact that $D$ is torsion
free,  we get
$$
d \omega (J X_1, J Y_1, J Z_1) - d \omega ({JX_2}, J Y_2, JZ_1)
  - d \omega ({JX_1}, J Y_2, JZ_2)
   - d \omega ({JY_1}, J Z_2, JX_2),
  $$
where $\omega$ is the K\"ahler form of $(J,g)$.

If $(\{J_{\alpha} \}, \tilde g)$ is HKT,  the above expression
vanishes for any $X_i, Y_i, Z_i \in \mathfrak g$. In particular,
setting $X_2 = Y_2 = 0$, we obtain that $\omega$ is closed, that is,
$({\mathfrak g} , J,\, g)$ is K\"ahler. Conversely, if $d \omega
=0$, the hyper-Hermitian structure   $(\{J_{\alpha} \}, \tilde g)$
is HKT. \end{proof}

\smallskip

\begin{remark}  We recall from \cite[Corollary 4.3]{BD} that the
Obata connection $\tilde \nabla^O$ of the HKT  Lie  algebra $(T _{D} {\mathfrak g}, \{ J_{\alpha } \}, \tilde g)$ is given by:
$$
\tilde \nabla^O_{(X_1, X_2)} (Y_1, Y_2) = (D_{X_1} Y_1, D_{X_1}
Y_2).
$$
and therefore is flat.
\end{remark}

Theorem  \ref{construction2} can be applied to any of the
$4$-dimensional K\"ahler Lie algebras classified in \cite{Ov}. Next,
we illustrate our construction in two particularly interesting
examples.

\begin{example}\label{e(2)} {\rm Let $\frak g = \R e_1 \oplus \frak e (2)$ be a trivial central
extension of $\frak e (2)$, the Lie algebra of the isometry group of
the Euclidean plane. We fix a basis $\{e_2, e_3,e_4\}$ of $\frak e
(2)$ with non-zero Lie brackets:
$$ [e_2,e_3]= e_4, \qquad \quad [e_2,e_4]= -e_3,
$$
and an inner product $g$ on $\frak g$ given by:
$$ g= \sum_{i=1}^4 (e^i)^2.$$
Let $J$ be the following complex structure on $\frak g$:
$$ Je_1=e_2, \qquad\quad Je_3=e_4.
$$
It turns out that $(\frak g, J, g)$ is K\"ahler. We define a flat
torsion-free connection $D$ on~$\frak g$ as follows:
$$ D_{e_1}=\begin{pmatrix}1&0& &  \\
0&1& &  \\         & & 0& 0  \\ & & 0& 0 \end{pmatrix},\quad
 D_{e_2}=\begin{pmatrix}0&-1& &  \\
1&0& &  \\         & & 0& -1  \\ & & 1& 0 \end{pmatrix},\quad
D_{e_3}=D_{e_4}=0.
$$
Since $D$ also satisfies $DJ=0$,  Theorem~\ref{construction2}
implies that $(T_D \, \mathfrak g, \{ J_{\alpha} \}, \tilde g)$ is
HKT.

 We show next that the previous HKT structure is weak by computing the torsion $3$-form
 $c$, which turns out to be non-closed.

The hypercomplex structure $\{ J_{\alpha} \}$ on $T_D \, {\mathfrak
g}$ is given by
\begin{align*}
J_1 e_1 &= e_2,&   J_1 e_3 &= e_4,&  J_1 e_5& = -e_6,&  J_1 e_7 &= -e_8,\\
J_2 e_1 &= -e_5,&  J_2 e_2 &= -e_6,&  J_2 e_3& = -e_7,&  J_2 e_4& =
-e_8,
\end{align*}
and the K\"ahler form of $(J_1, \tilde g)$ is $$\omega_1=
e^{12}+e^{34}-e^{56}-e^{78}.
$$ Then the torsion $c$ of the HKT structure and its exterior derivative $dc$ are given by
$$c=-J_1 d\omega _1=2\, e^{256}  , \qquad \quad  dc = - 4\,  e^{1256},$$
hence $c$ is not closed. It turns out that $c$ is co-closed, or,
equivalently, the Ricci tensor of the Bismut connection is symmetric
and that the metric $\tilde g$ is  not balanced since  $d( \omega_1
\wedge  \omega_1  \wedge \omega_1) \neq 0$. However, the Lee form
$\theta= 2e^1$ is closed,  therefore $\tilde g$ is conformally
balanced on the corresponding simply connected $8$-dimensional
solvable Lie group.}

\end{example}

\begin{example}
{\rm It is well known that the complex hyperbolic space $SU(2,1)/S\left( U(2)\times
U(1)\right)$ with its standard K\"ahler  structure
 admits a simply connected solvable
Lie group $G$ acting simply transitively by holomorphic isometries.
The $4$-dimensional Lie algebra  $\frak g$ of $G$ has a basis $e_1,
\dots , e_4$ with non-zero Lie brackets:
\[
[e_1, e_4] = - \frac12 \, e_1,  \quad [e_2, e_4] = - \frac12 \,
e_2,\quad [e_1, e_2] = e_3, \quad [e_3, e_4] = - e_3 .\]  The
K\"ahler  structure on $SU(2,1)/S\left( U(2)\times U(1)\right)$
induces a complex structure $J$ and an inner product $g$ on
$\mathfrak g$ given by:
\[ J e_1 = e_2, \qquad J e_3 = - e_4 , \qquad  g= \sum_{i=1}^4 (e^i)^2. \]
Therefore,   $(\frak g, J,g)$ is K\"ahler non-flat. Consider the
flat torsion-free connection $D$ on~$\frak g$ defined by $D_{e_3}=0$
and:
$$ D_{e_1}=\begin{pmatrix} & & 0&0  \\
 & & 0&  0\\   0      & \frac12& &  \\ -\frac12& 0& &
\end{pmatrix},\qquad
 D_{e_2}=\begin{pmatrix} & &0& 0 \\
 & & 0& 0 \\         -\frac12      &0 & &   \\ 0&-\frac12 & &
\end{pmatrix},\ \qquad D_{e_4}= \begin{pmatrix}\frac12&0& &  \\
0&\frac12& &  \\         & & 1& 0  \\ & & 0& 1 \end{pmatrix}.
$$
   Since $DJ=0$, Theorem~\ref{construction2}
implies that $(T_D \, \mathfrak g, \{ J_{\alpha} \}, \tilde g)$ is
HKT, where $\{ J_{\alpha} \}$ is defined as in \eqref{hk}. The
K\"ahler form of $(J_1, \tilde g)$ is given by: $$\omega_1=
e^{12}-e^{34}-e^{56}+e^{78},
$$
with corresponding  torsion $3$-form $c$:\[ c=-J_1 d\omega _1=-
\frac12 \, e^{268} - \frac12 \, e^{158} + 2 \, e^{378} + \frac12 \,
e^{167} - \frac12 \, e^{257} - e^{356}.\] It can be checked that $c$
is not closed (hence the HKT structure is weak) and the metric
$\tilde g$ is not balanced.
However, the Lee form $\theta = -3e^4$ is closed, hence exact on the corresponding simply connected $8$-dimensional solvable Lie group $T_D\, G$.
Therefore, the left invariant metric induced by $\tilde g$ on $T_D\, G$ is conformally balanced.} \end{example}

\

\end{document}